\documentclass[11pt]{amsart}
\usepackage{amssymb, amsthm, amscd, epsfig, graphicx,
enumerate, stmaryrd, amsmath, graphs, booktabs, hhline, xcolor, mathtools, float} 

\newtheorem{thm}{Theorem}[section]
\newtheorem{lem}[thm]{Lemma}
\newtheorem{prop}[thm]{Proposition}

\theoremstyle{definition}

\newtheorem{defn}[thm]{Definition}

\theoremstyle{remark}
\newtheorem{rem}[thm]{Remark}

\numberwithin{equation}{section}

\newcommand{\be}{\beta}

\newcommand{\ep}{\varepsilon}

\newcommand{\Si}{\Sigma}

\newcommand{\va}{\varphi}

\newcommand{\x}{\times}

\newcommand{\Z}{\mathbb Z}
\newcommand{\N}{\mathbb N}

\newcommand{\R}{\mathbb R}

\newcommand{\RP}{{\mathbb R}{P}}

\newcommand{\del}{\partial}

\newcommand{\co}{\colon\thinspace}

\renewcommand{\int}{\mathrm {int}}

\newcommand{\cl}{\mathrm {cl} \thinspace}

\begin{document}
\mathsurround=1pt 
\title{Concordance  of decompositions given by defining sequences}




\author{Boldizs\'ar Kalm\'{a}r}

\email{boldizsar.kalmar@gmail.com}


\begin{abstract}
We study the concordance and cobordism of decompositions associated with  defining sequences and we relate 
them to some invariants of toroidal decompositions and to  the cobordism of homology manifolds. 
 These decompositions are often wild Cantor sets and they arise as 
   nested intersections of knotted solid tori. 
 We show that  there are at least uncountably many 
 concordance classes of such decompositions in the $3$-sphere.
\end{abstract}

\maketitle


\section{Introduction}

We study equivalence classes of decompositions of $S^3$ and also decompositions of other manifolds.
These decompositions are given by toroidal defining sequences 
(we use the term {toroidal} for a subspace of an $n$-dimensional manifold 
being homeomorphic to the disjoint union of finitely many copies of $S^{n-2} \x D^2$) 
although more generally
it would be possible to get similar results by considering  handlebodies instead of solid tori in the 
defining sequences. 
The problem of classifying decompositions was studied by many authors.
By \cite{Sh68b} so-called Antoine decompositions in $\R^3$ are equivalently embedded   if and only if
their toroidal defining sequences can be mapped into each other by homeomorphisms of the stages.
More generally \cite{ALM68}   
for a  decomposition $G$ of $\R^3$ given by an arbitrary  
defining sequence made of handlebodies
the homeomorphism type of the pair  $(\R^3 / G, \cl \pi_G ( H_G))$, where
$\pi_G$ is the decomposition map and 
$H_G$ is the union of the non-degenerate elements,
is determined by the homeomorphism types of the consecutive stages of the defining sequence of $G$.
By \cite{GRWZ11} two Bing-Whitehead decompositions of $S^3$ are equivalently embedded if and only if
the stages of the  toroidal defining sequences
are homeomorphic to each other after some number of iterations (counting only the Bing stages).
Decompositions given by defining sequences are upper semi-continuous
and many shrinkability conditions are known about them.
For example, Bing-Whitehead decompositions are shrinkable under some conditions \cite{AS89, KP14}
just like Antoine's necklaces, which are wild Cantor sets.
In \cite{Ze05} the maximal  genus of handlebodies being associated with a  defining sequence 
is used to study  Cantor sets. 

In the present paper we define the  concordance of decompositions (see Section~\ref{concdef}) 
which come with 
toroidal defining sequences. As for knots, slice decompositions 
play an important role in the classification: a decomposition is slice if 
each  component of a defining sequence is slice in a way that the $D^{n-1} \x D^2$ thickened slice disk stages are nested into each other.
Being concordant means the analogous concordance of the solid tori in the defining sequence
and this makes the well-known knot and link concordance invariants possible to apply
in order to distinguish between the concordance classes of such decompositions.
For example, we show that 
the concordance group of decompositions of $S^3$, where the defining sequences have some intrinsic properties, 
has at least uncountably many elements, see Theorem~\ref{uncount1}.
The uncountably many elements that we find are represented by Antoine's necklaces. 

Decompositions appear in studying  manifolds, where cell-like 
resolutions of homology manifolds \cite{Qu82, Qu83, Qu87, Th84, Th04} provide a tool of 
obtaining topological manifolds. 
Decompositions  also appear 
in the proof of the Poincar\'e conjecture in dimension four, see \cite{Fr82, FQ90, BKKPR21}, where
 a cell-like decomposition of a $4$-dimensional manifold yields 
 a decomposition space which is a topological manifold. 
In higher dimensions the decomposition space given by a cell-like decomposition of a compact topological manifold
 is a homology manifold being also a topological manifold if it satisfies the disjoint disk property \cite{Ed16}. 
 A particular result \cite{Ca78, Ca79, Ed80, Ed06} is that the double suspension of every integral homology $3$-sphere
is homeomorphic to $S^5$,  
that is  for every homology $3$-sphere $H$ 
there is a cell-like decomposition $G$ of $S^5$ such that 
the decomposition space is the  homology manifold $\Si^2 H$
 and since $\Si^2 H$ satisfies  the disjoint disk property,
 the decomposition $G$ is shrinkable (and this implies that the decomposition space is $S^5$).  

Beside concordance, 
 we also define and study another equivalence relation, which is 
 the cobordism of decompositions, see Definition~\ref{borddecomp}
and Section~\ref{computebord}. This yields a cobordism group, which has
a natural homomorphism into 
the  cobordism group  of homology manifolds \cite{Mi90, Jo99, JR00}. We study  how homological manifolds  are related to 
the cobordism group of cell-like decompositions via taking the decomposition space. 
It turns out that 
every such decomposition space is cobordant to a topological manifold in the cobordism group of homology manifolds and
they generate a subgroup isomorphic to the cobordism group of topological manifolds, 
 see Proposition~\ref{bordkep}. 
 Often we state and prove our results only for unoriented cobordisms but all the arguments
 obviously work for the oriented cobordisms as well giving the corresponding results. 

The paper is organized as follows.
In Section~\ref{prelim}
we give some basic lemmas and the definitions of the most important notions and 
in Section~\ref{results}
we state and prove our main results.

The author  would like to thank the referee for the helpful comments, which improved the paper.

\section{Preliminaries}\label{prelim}


\subsection{Cell-like decompositions}

Throughout the paper we suppose that 
if $X$ is a compact manifold with boundary and 
$Y$ is a compact manifold with corners, then 
an embedding $e \co Y \to X$ is such that 
the corners of $Y$ are mapped into $\del X$ and the pairs of boundary components near  the corners  of $Y$ are 
mapped into $\int X$ and into $\del X$, respectively.
We also suppose that $e(\int Y) \subset \int X$. 
 If $Y$ has no corners, then 
$\del X \cap e(Y) = \emptyset$. 
 We generalize the notions of defining sequence, cellular set and cell-like set  in the obvious way  
for manifolds with boundary as follows. Recall that a decomposition of a topological space  $X$ is 
 a collection of pairwise disjoint subsets of $X$ whose union is equal to $X$. 

\begin{defn}[Defining sequence for a subset]\label{defseq}
Let $X$ be an  $n$-dimensional manifold with possibly non-empty boundary.
 A \emph{defining sequence}  for a subset $C \subset X$ 
 is a sequence 
 $$c \co \N \to \mathcal P(X)$$
 $$C_0, C_1, C_2, \ldots, C_n, \ldots$$ of compact  $n$-dimensional 
 submanifolds-with-boundary  possibly with corners 
 in $X$ such
 that 
 \begin{enumerate}
 \item
  every
$C_{n+1}$ has a neighbourhood $U$ such that  
$U \subset C_n$,  
\item
in every component of $C_n$ there is a component of $C_{n+1}$,
\item
$\cap_{n=0}^{\infty} C_n = C$ and
\item
if $\del X \neq \emptyset$, then 
there is an $\ep > 0$ such that 
$\del X \x [0, \ep)$ is  a collar neighbourhood 
of $\del X$ and for every $C_n$ such that 
 $C_n \cap \del X \neq \emptyset$
 we have $C_n \cap (\del X \x [0, \ep)) = (C_n \cap \del X) \x [0, \ep)$. 
\end{enumerate}
 A decomposition of $X$ defined by the  defining sequence $c$  is the triple 
 $(X, \mathcal D, C)$, where 
 $C = \cap_{n=0}^{\infty} C_n$ and 
 the elements of $\mathcal D \subset \mathcal P(X)$ 
 are 
\begin{enumerate}
\item
the  connected components of $C$ and 
\item
the points in $X -C$.
\end{enumerate}
We denote the decomposition map by $\pi$. 
\end{defn}

Observe that for a decomposition $(X, \mathcal D, C)$ the set $C$ is non-empty and each of the non-degenerate elements is a subset of $C$. 
There could be singletons in $C$ as well. 
For example in the case of an Antoine's necklace there are no non-degenerate elements, we  choose
 $C$ to be the Cantor set Antoine's necklace itself and so  $C$ consists of singletons. 
Every decomposition  defined by some defining sequence is upper semi-continuous. 
A decomposition $\mathcal D$  of a manifold induces a decomposition on its boundary by intersecting the decomposition elements
with the boundary. 
The decomposition of the boundary $\del X$  induced  by a defining sequence in $X$ is upper semi-continuous.
 This induced decomposition is given by an induced  defining sequence $C_{n} \cap \del X$
  if $D_{n, k} \cap \del X \neq \emptyset$ for  every component $D_{n, k}$ of each $C_n$. 
 If all $C_n$ in a defining sequence are connected, then  $\cap_{n=0}^{\infty} C_n$ is connected.

\begin{defn}[Cell-like set]
A compact subset $C$ of a metric space $X$  is \emph{cell-like} if 
for every neighbourhood $U$ of $C$ there is a neighbourhood $V$ of $C$ in $U$ such that
the inclusion map $V \to U$ is homotopic in $U$ to a constant map.
A decomposition is called cell-like if
each of its decomposition elements is  cell-like.
\end{defn}

Cell-like sets given by defining sequences are connected 
because if the connected components could be separated by open neighbourhoods, then   
 a  homotopy could not deform the set into one single point in the neighbourhoods.


A  space $X$ is \emph{finite dimensional} if for every open cover $\mathcal U$ of $X$ there exists
a refinement $\mathcal V$ of $\mathcal U$ such that no points of $X$ lies in more 
than $K_X$ of the elements of $\mathcal V$, where $K_X$ is a constant depending only on $X$.

\begin{lem}\label{dimdecomp}
Let $\mathcal D$ be a decomposition of a manifold  $X$ possibly with non-empty boundary
given by a defining sequence.
 Then the  decomposition space $X / \mathcal D$ 
 is finite dimensional.
\end{lem}
\begin{proof}
If $X$ has no boundary, then the statement follows from 
 Theorem~2 and Proposition~3 in \cite[Chapter~34]{Da86}.
 If $X$ has non-empty boundary, then 
 the argument is also similar. 
\end{proof}

\subsection{Homology manifolds}

Recall that 
a metric space $Y$ is an \emph{absolute neighbourhood retract}  (or \emph{ANR} for short)
if 
for every metric space $Z$ and embedding $i \co Y \to Z$  such that 
$i(Y)$ is closed there is a neighbourhood $U$ of $i(Y)$ in $Z$ which retracts onto $i(Y)$, that is
$r|_{i(Y)} = \mathrm {id}_{i(Y)}$ for some map $r \co U \to i(Y)$.
It is a fact that every manifold is an ANR.
A space is called a  \emph{Euclidean neighbourhood retract}  (or \emph{ENR} for short)
if it can be embedded into a Euclidean space as a closed subset  so that it is a retract of 
some of its  neighbourhoods. 
It is well-known that 
a space is an ENR if and only if it is a locally compact, finite dimensional, separable  ANR.

%


\begin{defn}[Homology manifold]
Let $n \geq 0$ and let $X$ and $Y$ be finite dimensional  ANR spaces, where
$Y$ is a closed subset of $X$. 
Suppose that for every $x \in X$ we have
\begin{enumerate}
\item
$H_k ( X, X - \{ x \} ) = 0$ for $k \neq n$ and 
\item
$H_n ( X, X - \{ x \} )$ is isomorphic to  $\Z$ if $x \in X - Y$ and  it is isomorphic to  $0$ if $x \in Y$.
\end{enumerate}
Then $X$ is an $n$-dimensional  \emph{homology manifold}. 
The set of points $x \in Y$  
are the \emph{boundary points of $X$} and the set $Y$  is denoted by $\del X$.
A homology manifold is called \emph{closed} if it is compact and has no boundary.
\end{defn}

Since  locally compact and separable homology manifolds  are ENR spaces, a locally compact and separable
 homology manifold is called an \emph{ENR homology manifold}.
In \cite{Mi90} it is proved that
for $n \geq 1$ and for  every compact and locally compact $n$-dimensional homology manifold  $X$
 the set of boundary points $\del X$ 
 is an $(n-1)$-dimensional  homology manifold.


Sometimes a space $X$ without the ANR property but having 
$H_k ( X, X - \{ x \} ) = 0$ for $k \neq n$ and 
$H_n ( X, X - \{ x \} ) = \Z$ in the sense of \v{C}ech homology is also called a homology manifold. These spaces arise as 
quotient spaces of acyclic decompositions of topological manifolds 
\cite{DW83} while ANR homology manifolds are often homeomorphic to 
quotients of cell-like decompositions \cite{Qu82, Qu83, Qu87}.

In the case of cell-like decompositions the decomposition spaces are homology manifolds if they are finite dimensional essentially because 
of the Vietoris-Begle theorem \cite[Theorem~0.4.1]{DV09}. In more detail, we will use the following. 
Let $X'$ be a compact $n$-dimensional  manifold with possibly non-empty boundary, let
$Y$ be  $\del X' \x [0,1]$ and attach $Y$ to $X'$ as a collar  to get a manifold $X$.

\begin{lem}\label{decomphomolmanif}
Let $\mathcal D'$ be a cell-like decomposition of $X'$ given by a defining sequence such that 
$X'$ contains a small open set (intersecting the possibly non-empty boundary) which consists of singletons. 
 Suppose that 
 the induced decomposition  on $\del X'$ is cell-like and it is given by the induced defining sequence. 
 Suppose that 
 in $Y$  a cell-like decomposition $\mathcal E$ is given, where
 $\mathcal E$ is the product of the  decomposition 
 induced by $\mathcal D'$ on $\del X'$ and the trivial decomposition of $[0,1]$. 
  Denote by $\mathcal D$ the resulting decomposition on $X$. 
Then $X/\mathcal D$ 
is an $n$-dimensional ENR homology manifold with possibly non-empty boundary. 
 The boundary points of $X/\mathcal D$ are exactly the points of the ENR homology manifold 
$\pi (\del X)$.
\end{lem}
\begin{proof}
We have to show that 
the quotient space 
$$X / \mathcal D$$
is an $n$-dimensional  homology manifold with boundary the homology manifold
$\pi (\del X)$.
Take the closed manifold $$X  \cup_{\va} X,$$
where
$\va \co \del X  \to \del X$
is the identity map.

The decomposition space $X' / \mathcal D'$
(that is the part of the decomposition space $X / \mathcal D$ which is obtained from 
$X'$)   
is finite dimensional by Lemma~\ref{dimdecomp}. 
The doubling of the decomposition $\mathcal D$ on 
$X  \cup_{\va} X$
yields
a finite dimensional quotient space, we get this by using 
estimations for the covering dimension, see \cite{HW41} and \cite[Corollary~2.4A]{Da86}.
So the decomposition space $P$ obtained by factorizing 
$X  \cup_{\va} X$
by the double of $\mathcal D$ is a closed  finite dimensional homology manifold by  \cite[Proposition~8.5.1]{DV09}. 
Since a small neighbourhood of a singleton results  an open set in $P$ homeomorphic to $\R^{n}$, it is $n$-dimensional.
We obtain the space $X / \mathcal D$
by cutting $P$ into two pieces along  $\pi(\del X)$. Because of a similar argument  
the space $\pi(\del X)$ is a closed  $(n-1)$-dimensional homology manifold. 
The set $\pi(\del X)$ is closed in the decomposition space $X / \mathcal D$ since
$\mathcal D$ is upper semi-continuous and $\del X$ is closed. 
Also, the homology group  $H_n(X / \mathcal D; X / \mathcal D - \{p\} )$ is equal to 
$0$ for every $p \in \pi(\del X)$. So $\pi(\del X)$ is the boundary of $X / \mathcal D$. 

Moreover the space $X / \mathcal D$ is 
 a locally compact separable metric space because $X$ is so. By \cite[Corollary~7.4.8]{DV09}
the space $X/\mathcal D$ is an ANR so it follows that it is an ENR. The same holds for $\pi(\del X)$. 
\end{proof}

\begin{defn}[Cobordism of homology manifolds]
The  closed $n$-dimensional  
homology manifolds $X_1$ and $X_2$ are \emph{cobordant} 
if there exists  a compact $(n+1)$-dimensional homology manifold $W$ such that 
$\del W$ is homeomorphic to the disjoint union of $X_1$ and $X_2$.
The induced cobordism group (the group operation is the disjoint union) is 
denoted by $\mathfrak N_n^H$. In a similar way the induced oriented cobordism group is denoted by $\Omega_n^H$.
\end{defn}
Note that the connected sum of homology manifolds does not always exist. 
Analogously let $\mathfrak N_n^E$ and $\Omega_n^E$ denote the cobordism group and
oriented cobordism group 
of ENR homology manifolds (the cobordisms are also ENR), respectively.

Almost all oriented cobordism groups $\Omega_n^H$ are computed \cite{BFMW96, Jo99, JR00}:
\[
\Omega_n^H = 
\left \{
\begin{array}{cc}
\Z  & \mbox{if $n = 0$}   \\
0  &   \mbox{if $n = 1, 2$}   \\
 \Omega_n^{TOP}[8\Z  +1]   &     \mbox{if $n \geq 6$},
\end{array}
\right.
\]
where $\Omega_n^{TOP}$ denotes the cobordism group of topological manifolds
 and the group $$\Omega_n^{TOP}[8\Z  +1]$$ denotes 
 the group of finite linear combinations 
 $\sum_{i \in 8\Z +1} \omega_i i$ of cobordism classes of topological manifolds.
 By \cite[Corollary~4.2]{Ma71} the oriented cobordism group of manifolds $\Omega_n$ is always a subgroup
of $\Omega_n^H$.

A \emph{resolution} of a homology manifold  $N$ 
is a topological manifold $M$ and a cell-like decomposition of $M$ 
such that the decomposition space is homeomorphic to the homology manifold $N$,
the quotient map $\pi$ is proper 
 and
$\pi^{-1} (\del N) = \del M$. 
By \cite{Qu82, Qu83, Qu87} homology manifolds are resolvable if a local obstruction is equal to $1$, more precisely we have the following.
\begin{thm}[\cite{Qu82, Qu83, Qu87}]\label{resol}
For every $n \geq 4$ and every non-empty connected $n$-dimensional ENR  homology manifold $N$ there is an integer local obstruction $i(N) \in 8\Z + 1$
such that 
\begin{enumerate}[\rm (1)]
\item
if $U \subset N$ is open, then $i(U) = i(N)$,
\item
if $\del N \neq \emptyset$, then 
$i(\del N) = i(N)$, 
\item
$i(N \x N_1) = i(N) i(N_1)$ for any other homology manifold $N_1$,
\item
if $\dim N = 4$ and $\del N$ is a manifold, then there is a resolution if and only if $i(N)=1$ and 
\item
if $\dim N \geq 5$, then there is a resolution if and only if $i(N)=1$.
\end{enumerate}
\end{thm}
By \cite{Th84, Th04} a closed  $3$-dimensional ENR homology manifold $N$ is resolvable if 
its singular set has general position dimension less than or equal to one, that is
any map of a disk into $N$ can be approximated by one whose image meets the singular set 
 (i.e.\ the set of non-manifold points) 
of $N$ in a $0$-dimensional set.

\begin{lem}\label{resolcob}
Let $M_{1}$ and $M_{2}$ be two closed $n$-dimensional manifolds, where $n \geq 4$.
If both of them are resolutions of the ENR homology manifold $N$, then 
$M_1$ and $M_2$ are cobordant as manifolds.
\end{lem}
\begin{proof}
If there are two resolutions $f_1 \co M_1 \to N$ and $f_2 \co M_2 \to N$ of 
 a closed $n$-dimensional homology manifold $N$, then 
 as in the proof of \cite[Theorem~2.6.1]{Qu82}
 take a resolution 
 $$Y \to X_{f_1} \cup X_{f_2}$$ 
 of the double mapping cylinder $X_{f_1} \cup X_{f_2}$  
 of the maps $f_1$ and $f_2$ 
  by applying \cite[Theorem~1.1]{Qu83} and \cite{Qu87}. 
 This resolution exists because $X_{f_1} \cup X_{f_2}$ is an $(n+1)$-dimensional ENR homology manifold and 
 $i( X_{f_1} \cup X_{f_2} )  =1$.
 Let 
 $$X_{f_1} \cup X_{f_2} \to N \x [-1,1]$$
be the natural map of the double mapping cylinder 
onto $N \x [-1,1]$, where the target $N$ of the two mapping cylinders is mapped 
onto $N \x \{ 0 \}$. 
 
It follows that the composition 
 $$Y \to X_{f_1} \cup X_{f_2} \to N \x [-1,1]
$$ 
 is a resolution, moreover by \cite[Theorem~1.1]{Qu83}
 the cell-like map 
 $Y \to X_{f_1} \cup X_{f_2}$ can be chosen 
so that 
 it is a homeomorphism over the boundary hence
 $Y$ is a cobordism  between $M_1$ and $M_2$.
\end{proof}

%

\subsection{Concordance and cobordism of decompositions}\label{concdef}

We will study decompositions given by defining sequences $C_0, C_1, C_2, \ldots$ such that 
each $C_n$ is a disjoint union of solid tori. 
 We remark  that more generally all the following notions work for  decompositions whose stages are handlebodies instead of just tori.
In a closed $n$-dimensional manifold $M$
 instead of decompositions $(M, \mathcal D, A)$
we will consider decompositions with 
some thickened link which contains the set $A$ so 
 in the following
 a decomposition in $M$ is a quadruple $(M, \mathcal D, A, L)$, where $L \subset M$ is the thickened link and $A \subset L$.
For example an Antoine's necklace is situated inside an unknotted 
solid torus while it can be knotted in many different ways in the solid torus.

\begin{defn}[Concordance of decompositions]\label{concordecomp}
Let $M_1$ and $M_2$ be closed $n$-dimensional manifolds. 
The  decompositions
 $(M_1, \mathcal D_1, A, L_1)$ and $(M_2, \mathcal D_2, B, L_2)$
are \emph{cylindrically related}
if  
there exist
 toroidal  defining sequences
$C_0, C_1, C_2, \ldots$ for $A$ and $D_0, D_1, D_2, \ldots$ for $B$ and 
there exists 
a defining sequence $E_0, E_1, E_2, \ldots$ for a   decomposition $\mathcal E$ 
of a compact $(n+1)$-dimensional manifold $W$ such that 
\begin{enumerate}
\item
$C_0 = L_1$ and $D_0 = L_2$, 
\item
$\del W = M_1 \sqcup M_2$,
\item
each $E_i$ is homeomorphic to  $C_i \x [0,1]$ and 
\item
each $E_i$  bounds the components of $C_i \subset M_1$ and $D_i \subset M_2$ that is  
 $C_i \x \{ 0 \}$ corresponds to $C_i$ and $C_i \x \{ 1 \}$ corresponds to $D_i$. 
\end{enumerate}
Two decompositions $(M_1, \mathcal D_1, A, L_1)$ and $(M_2, \mathcal D_2, B, L_2)$
 are \emph{concordant} if 
 there exist closed $n$-dimensional manifolds
 $M_1', \ldots, M_k'$ and  decompositions $(M_i', \mathcal D_i',  A_i', L_i')$ for every $i = 1, \ldots, k$ 
  such that 
  \begin{enumerate}
  \item
  $(M_1, \mathcal D_1, A, L_1)$ is cylindrically related to 
  $(M_1', \mathcal D_1',  A_1', L_1')$, also for $i = 1, \ldots, k-1$ every 
  $(M_i', \mathcal D_i',  A_i',  L_i')$ is cylindrically related to 
  $(M_{i+1}', \mathcal D_{i+1}',  A_{i+1}',  L_{i+1}')$
 and 
 $(M_k', \mathcal D_k',  A_k',  L_k')$ is cylindrically related to
 $(M_2, \mathcal D_2, B, L_2)$ and 
 \item
for each $A_i' \subset M_i'$, where $i = 1, \ldots, k$,  the two toroidal defining sequences 
 $C_{i, 0}', C_{i, 1}', \ldots$ and  $C_{i, 0}'', C_{i, 1}'', \ldots$  in $M_i'$ 
 appearing in these successive cylindrically related decompositions
 are such that the $0$-th stages 
 $C_{i, 0}'$ and $C_{i, 0}''$ are equal as subsets of $M_i'$.
\end{enumerate}
 Being concordant is an equivalence relation 
 and the equivalence classes are called \emph{concordance classes}.
 \end{defn}
 
Hence being concordant implies that the two decompositions
 are in the same equivalence class of 
the equivalence relation generated by being cylindrically related, that is 
the two decompositions can be connected by a finite
number of cylindrically related  decompositions. 
 Being concordant also implies that 
 the $0$-th stages of two toroidal defining sequences for the two decompositions
 are connected by a single concordance in the usual sense. 
Clearly in the definition 
each $E_i$ intersects  some fixed collar of $\del W$ as the defining sequence in (4) of Definition~\ref{defseq}.
 The  concordance classes form a 
commutative semigroup under the operation ``disjoint union". Moreover this 
semigroup is a monoid because the neutral element 
is the ``empty manifold'', that is the empty set $\emptyset$.
To have a more meaningful neutral element we define the following.

\begin{defn}[Slice decomposition]\label{slicedecomp}
Let $M$  be a closed $n$-dimensional manifold and 
let $(M, \mathcal D, A, L)$  
be a   decomposition of $M$ such that 
there exists a toroidal  defining sequence $C_0, C_1, C_2, \ldots$ with $C_0 = L$ for  $A$. 
Then $(M, \mathcal D, A, L)$ is \emph{slice} if 
it is concordant to a decomposition $(M', \mathcal D', A', L')$
with defining sequence $C_0', C_1', C_2', \ldots$ with $C_0' = L'$ 
 such that
there exists
 a defining sequence $E_0, E_1, E_2, \ldots$ for a   decomposition $\mathcal E$ 
of the  $(n+1)$-dimensional manifold $M' \x [0,1)$, where 
each $E_i$ consists of finitely many $D^{n-1} \x D^2$ bounding the torus components $S^{n-2} \x D^2$ of $C_i' \subset M' \x \{0\}$. 
 \end{defn}
 
 Analogously 
to Definitions~\ref{concordecomp} and \ref{slicedecomp}, we 
define the \emph{oriented concordance} of decompositions by requiring 
all the manifolds  to be oriented in the usual consistent way, 
in this way we also get  a corresponding monoid.   
Observe that the set of concordance classes of slice decompositions is a submonoid 
of the monoid of concordance classes of decompositions.  
 To obtain a group 
we factor out the concordance classes by the classes represented by the slice decompositions
and also by the classes of the form $$[(M, \mathcal D, A, L)] + [(-M, \mathcal D, A, L)],$$
where $-M$ denotes the opposite orientation. Observe that all these classes form a submonoid. 

\begin{defn}[Decomposition concordance group]
Define the relation $\sim$ on the set of concordance classes of decompositions by the following rule:  
$a \sim b$ exactly if
there exist  slice decompositions   $s_1$ and $s_2$ 
and decompositions $(M, \mathcal D, A, L)$  and 
$(M', \mathcal D', A', L')$   
such  that 
$$
 a + [s_1] +  [( M, \mathcal D, A, L )] + [(- M, \mathcal D, A, L )] = 
  b + [s_2] + 
 [( M', \mathcal D', A', L')] + [(- M', \mathcal D', A', L')].
$$
The relation $\sim$ is a congruence and we obtain a commutative group by factoring out by this 
congruence. We call this group the \emph{oriented decomposition concordance group} 
and denote it by $\Gamma_n$.
\end{defn}

If we confine the closed $n$-dimensional  manifolds to $S^n$ and the cobordisms to $S^n \x [0,1]$, then we obtain something similar to the classical link concordance.
For the convenience of the reader we repeat the definitions. 

\begin{defn}[Concordance group of decompositions in $S^n$]\label{concordecompS}
Let $(S^n, \mathcal D_1, A, L_1 )$ and $(S^n, \mathcal D_2, B, L_2 )$
be  decompositions of $S^n$ in the complement of $\infty$.
 They  
are \emph{cylindrically related}
if  
 there exist
toroidal  defining sequences
$C_0, C_1, C_2, \ldots$  for $A$ and $D_0, D_1, D_2, \ldots$   for $B$ and 
there exists 
a defining sequence $E_0, E_1, E_2, \ldots$ for a  decomposition $\mathcal E$ 
of the compact $(n+1)$-dimensional manifold $S^n \x [0,1]$ in the complement of $\{ \infty\} \x [0,1]$ such that 
\begin{enumerate}
\item
$C_0 = L_1$ and $D_0 = L_2$, 
\item
each $E_i$ is homeomorphic to  $C_i \x [0,1]$ and 
\item
each $E_i$  bounds the components of $C_i \subset S^n \x \{0\}$ and $D_i \subset S^n \x \{1\}$. 
\end{enumerate}
Two decompositions are \emph{concordant} if 
\begin{enumerate}
\item
they are in the same equivalence class of 
the equivalence relation generated by being cylindrically related 
so  the two decompositions can be connected by a finite
number of cylindrically related  decompositions
and 
\item
the $0$-th stages of the defining sequences 
appearing in this sequence of  cylindrically related  decompositions
 are concordant as thickened links in the usual sense.
\end{enumerate}
The obtained  equivalence classes are called \emph{concordance classes}. 
 If two decompositions of $S^n$ are given by defining sequences, 
then in the connected sum  (at  $\infty$)  of the two $n$-spheres the ``disjoint union''   
induces 
 a commutative semigroup operation on the set of concordance classes. 
 Then 
 by factoring out  by the 
submonoid of classes of slice decompositions and 
classes of the form $[(S^n, \mathcal D, A, L)] + [(-S^n, \mathcal D, A, L)]$
we
 get a group called the \emph{decomposition concordance group in $S^n$}.
  We denote this group by $\Delta_n$.
\end{defn}

 For example, the Whitehead decomposition in $S^3$ 
is slice \cite{Fr82} and the Bing decomposition in $S^3$ is also slice
because the Bing double of the unknot is slice. Observe that 
the Bing decomposition $(S^3, \mathcal B, C)$ has only singletons, where $C$ is a wild Cantor set.
As another example, a defining sequence in $S^3$ given by the replicating pattern of a solid torus and inside of it 
 a  link made of a sequence of ribbon knots linked with each other circularly 
 can yield a  slice decomposition. 


Since being concordant  implies  that the two decompositions can be connected by a finite
number of cylindrically related  decompositions, all 
invariants 
 of concordance classes defined through defining sequences 
  are invariant under choosing another defining sequence for the same decomposition (while leaving the $0$-th stage unchanged). 
For $n = 3$ in the following we restrict ourselves only to such toroidal 
defining sequences $C_0, C_1, C_2, \ldots$ of decompositions of the closed $n$-dimensional manifolds  
in Definitions~\ref{concordecomp}-\ref{concordecompS}
which 
satisfy 
the following  conditions:
\begin{defn}[Admissible defining sequences and decompositions]\label{adm}
Suppose 
\begin{enumerate}
\item
for $m \geq 1$ each $C_m$ has at least four components in a component of $C_{m-1}$ and 
each component $T$ of  $C_m$ 
is linked to exactly  two other components of $C_m$ in the ambient space $S^3$ with algebraic linking number non-zero
and the splitting number of $T$ and each of  the other components is equal to $0$, 
\item
for $m \geq 1$ the components $A_1, \ldots, A_k$ of  $C_m$ which are in a component $D$ of $C_{m-1}$ 
are linked in such a way that 
if a component $A_i$  is null-homotopic in a solid torus $T$ whose boundary is disjoint from all $A_i$, then 
all $A_i$  are in this solid torus $T$,
\item
$\cap_{m=0}^{\infty} C_m$ is not separated by and not contained in any 
 $2$-dimensional sphere 
$S$ for which 
$S \subset C_m$ for some $m$, 
\item
 every embedded circle in  the boundary of a component of $C_m$
which bounds no $2$-dimensional disk in this boundary
cannot be shrunk to a point in the complement of $\cap_{m=0}^{\infty} C_m$. 
\end{enumerate}
We call such defining sequences and decompositions \emph{admissible}. 
\end{defn}

\begin{prop}
In the connected sum  (at  $\infty$)  of two $3$-spheres the ``disjoint union''   as in Definition~\ref{concordecompS} 
of two admissible toroidal decompositions  is an  admissible toroidal decomposition.
\end{prop}
\begin{proof}
Checking the conditions (1)-(4) in Definition~\ref{adm} is obvious, details are left to the reader. 
\end{proof}

%

Then we denote the arising concordance  group in $S^3$ by  
$\Delta_3^a$. 
For example, Antoine's necklaces (or Antoine's decompositions) for $n = 3$ 
have defining sequences satisfying these conditions \cite{Sh68b}. 
We note that by \cite{Sh68b} their defining sequences also have  the property  of 
simple chain type, which means that 
the torus components are unknotted and they are linked like the Hopf link. 
 We have 
the natural group homomorphisms
$$
\Delta_3^a \to \Delta_3\mbox{\ \ \ \ \ and\ \  \ \ }\Delta_3 \to \Gamma_3$$
and also for arbitrary $n$ the group homomorphism
$$
\Delta_n \to \Gamma_n.$$
We will show that 
the number of elements of the group $\Delta_3^a$ is at least uncountable.


Now we define \emph{cobordism} of decompositions, where 
we restrict ourselves to \emph{cell-like} decompositions (not necessarily admissible)  at the cobordisms 
and at the representatives as well.

\begin{defn}[Cobordism of decompositions]\label{borddecomp}
Let $M_1$ and $M_2$ be closed $n$-dimensional manifolds and 
let $(M_1, \mathcal D_1, A)$ and $(M_2, \mathcal D_2, B)$
be cell-like decompositions  such that 
there exist
 toroidal  defining sequences
$C_0, C_1, C_2, \ldots$ for $A$ and $D_0, D_1, D_2, \ldots$ for $B$. 
Then $(M_1, \mathcal D_1, A)$ and $(M_2, \mathcal D_2, B)$
are \emph{coupled}
if  there exists 
a defining sequence $E_0, E_1, E_2, \ldots$ for a cell-like  decomposition $\mathcal E$ 
of a compact $(n+1)$-dimensional manifold $W$ such that 
\begin{enumerate}
\item
$\del W = M_1 \sqcup M_2$,
\item
each $E_i$ is homeomorphic to  the disjoint union 
of finitely many manifolds $P_j^{n-1} \x D^2$, $j = 1, \ldots m_i$, where all $P_j^{n-1}$ are  
compact $(n-1)$-dimensional manifolds and 
\item
each $E_i$  bounds the components of $C_i$ and $D_i$.
\end{enumerate}
We attach a collar $\del W \x [0,1]$ to $W$ along its boundary and extend the decomposition 
$\mathcal E$ to the collar by taking the product  of $\mathcal D_1$ and 
$\mathcal D_2$  with the trivial decomposition on $[0,1]$, respectively. We say that this extended manifold  
$W \cup (\del W \x [0,1])$ and its decomposition is a \emph{coupling} between $(M_1, \mathcal D_1, A)$ and $(M_2, \mathcal D_2, B)$.
 Finally, two decompositions are \emph{cobordant} if they are in the same equivalence class of 
the equivalence relation generated by being coupled. 
The generated equivalence classes are called \emph{cobordism classes}. 
\end{defn}

Clearly
each $E_i$ intersects  some fixed collar of $\del W$ as the defining sequence in (4) of Definition~\ref{defseq}.
 The cobordism classes  form a commutative group 
under the operation ``disjoint union". 
Denote this group by $\mathcal B_n$.

We will show that 
for a cobordism between arbitrary given cell-like decompositions $\mathcal D_{1, 2}$ as in Definition~\ref{borddecomp} if we 
 take the decomposition  space, then we get
 a group homomorphism 
into the cobordism group of homology manifolds.

\section{Results}\label{results}

\subsection{Computations in the concordance groups}

We are going  to define invariants  of  elements of
the group $\Delta_3^a$. 
 With the help of these invariants, we will show that 
 the group $\Delta_3^a$  has at least uncountably many elements.

\begin{defn}
For a given defining sequence 
$C_0, C_1, C_2, \ldots, C_n, \ldots$ in $S^3$  let $$n_{C_0, C_1, C_2, \ldots } = (n_0, n_1, n_2, \ldots )$$ be 
the sequence of the numbers of components of the manifolds $C_0, C_1, C_2, \ldots$.
\end{defn}

If two decompositions of $S^3$ as in Definition~\ref{concordecompS} are cylindrically related, then 
they have defining sequences 
$C_0, C_1, C_2, \ldots$ and $D_0, D_1, D_2, \ldots$   such that 
$$
n_{C_0, C_1, C_2, \ldots }  = n_{D_0, D_1, D_2, \ldots }. 
$$
By \cite[Theorem~3]{Sh68b}   for 
canonical defining sequences of an Antoine's necklace (or an  Antoine decomposition) 
the sequence $n_{C_0, C_1, C_2, \ldots }$ 
 uniquely exists (note that $C_0$ is only an unknotted solid torus which is not 
 appearing in \cite{Sh68b}).

\begin{prop}\label{compnumbsame}
 Let  $(S^3, \mathcal D, A, C_0)$ be an admissible
decomposition  and  let
  $C_0, C_1, C_2, \ldots$ and $D_0, D_1, D_2, \ldots$  be
  admissible defining sequences 
  for $(S^3, \mathcal D, A, C_0 )$, where we suppose that   $C_0 = D_0$. 
Then we have
$$n_{C_0, C_1, C_2, \ldots }  = n_{D_0, D_1, D_2, \ldots}.$$
\end{prop}
\begin{proof}
Suppose that 
$C_0, C_1, C_2, \ldots$ and $D_0, D_1, D_2, \ldots$ are admissible defining sequences for 
a decomposition $(S^3, \mathcal D, A, C_0)$ such that $C_0 = D_0$. 
Of course 
$$\cap_{n = 0}^{\infty} C_n = A = \cap_{n = 0}^{\infty} D_n.$$
We use an algorithm applied in \cite[Proof of Theorem~2]{Sh68b}. 
We restrict ourselves to 
one component of $C_0$ and to the components of the  defining sequences in it, the following argument
 works the same way for the other components. 
We can suppose that 
$\del C_1 \cap \del D_1$ is a closed $1$-dimensional submanifold of $S^3$.
 Suppose some  component $P$ of $\del C_1 \cap \del D_1$
bounds a  $2$-dimensional disk $Q \subset \del D_1$.
Also suppose that $P$ is an innermost component of $\del C_1 \cap \del D_1$ in 
$\del D_1$ so $\int Q \cap \del C_1 = \emptyset$. 
By (4) in Definition~\ref{adm} if $P$ does not bound a disk $Q'$ in $\del C_1$, then 
$P$ is not homotopic to constant in 
the complement of $A$ but then $P$ cannot bound the disk $Q \subset \del D_1$. 
Hence 
$P$ bounds a disk $Q' \subset \del C_1$ as well. 
Then the interior of the sphere $Q \cup Q'$ 
does not intersect $A$ 
because of (3) in Definition~\ref{adm}. 
So we can modify $C_1$ by pushing $Q'$  through the sphere $Q \cup Q'$ by a self-homeomorphism of the complement 
of $A$ 
and hence we obtain  fewer circles in the new $\del C_1 \cap \del D_1$.
After repeating these steps finitely many times we obtain 
a new $C_1$ such that $\del C_1 \cap \del D_1$ contains no circles which bound disks
 on $\del C_1 \cup \del D_1$. 
 Similarly, by further adjusting $C_1$ in the complement of $A$ as written on  \cite[page~1198]{Sh68b}
  in order to eliminate the circles in $\del C_1 \cap \del D_1$ which bound annuli
 we finally obtain a $C_1$ such that 
\begin{itemize}
\item
the intersection $\del C_1 \cap \del D_1$ is empty,
\item
no component of $C_1$ is disjoint from all the components of $D_1$ and vice versa,
\item
each component of $C_1$ is inside 
a component of $D_1$ or it contains some components of $D_1$.
\end{itemize} 

Then we can see that there is a bijection between the 
number of components of $C_1$ and $D_1$ because of the following.

If a component of $C_1$ is in $\int D_1$ and it is homotopic to constant 
in $\int D_1$, then all the other components of $C_1$ are in the same component of $\int D_1$ by (2) in Definition~\ref{adm}.
This would result that no part of $A$ is in other components of $D_1$, which would  contradict to 
(1) in Definition~\ref{adm}
 so 
no 
component of $C_1$ in $\int D_1$  is homotopic to constant 
in $\int D_1$. The same holds if we switch the roles of $C_1$ and $D_1$. 
This means that 
\begin{itemize}
\item
the winding number
of a component $T$ of $C_1$ in the component of $D_1$ which contains $T$ is not equal to $0$ and the same holds
for $D_1$ and $C_1$ with opposite roles. 
\end{itemize}

Furthermore 
suppose that $T$ is some component of $D_1$ and
 $T$  contains at least two  components $T_1$ and $T_2$ of $C_1$.
 Then 
   $T$ is linking with other component $T'$ of $D_1$ by (1) in Definition~\ref{adm} with algebraic 
   linking number non-zero.  
Let $T_3$ be a component of $C_1$ such that 
 $T_3 \subset T'$ or $T' \subset T_3$. 
If $T_3 \subset T'$, then 
$T_1$ and $T_2$ are linking with $T_3$ with algebraic linking number non-zero.
If $T' \subset T_3$, then 
$T$ is not in $T_3$ because for example 
$T_1$ cannot be in $T_3$. 
But then $T$ is linking with $T_3$ with algebraic linking number non-zero since the same holds for $T$ and $T'$. 
 So again we obtained that $T_1$ and $T_2$ are linking with $T_3$ with linking number non-zero.
Now, there is a $T''$ component of $D_1$
 which is linking with $T'$ with linking number non-zero and which is disjoint from all the previously mentioned tori 
 ($T', T'' \subset T_3$ is impossible because then both of $T', T''$ are linking with $T$ and also with each other 
 and this contradicts to (1) in Definition~\ref{adm}).
  Let $T_4$ be a component of $C_1$  such that 
 $T_4 \subset T''$ or $T'' \subset T_4$. There are a number of cases to check. 
 If $T_3 \subset T'$
  and  $T_4 \subset T''$, then 
  $T_3$ is linking with $T_4$. 
  If $T_3 \subset T'$ but $T'' \subset T_4$, then 
  since $T_4$ cannot contain 
  $T$ or $T'$, we have again that 
  $T_3$ is linking with $T_4$. 
  Finally, if $T' \subset T_3$, then since $T''$ cannot be in $T_3$, we have that 
  $T_4 \subset T''$ implies that $T_3$ and $T_4$ are linking and
  $T'' \subset T_4$ implies that 
 since  $T_4$ is disjoint from all the other tori, again $T_4$ is linking with $T_3$.
So we obtain that 
$T_1$ and $T_2$ are linking with $T_3$
 and $T_3$ is linking with $T_4$
 resulting that $T_3$ is linking with three other components of $C_1$ which contradicts to 
(1) in Definition~\ref{adm}.
Summarizing, we obtained the following.
\begin{itemize}
\item
The intersection $\del C_1 \cap \del D_1$ is empty,
\item
no component of $C_1$ is disjoint from all the components of $D_1$ and vice versa,
\item
every component of $C_1$ contains one component of $D_1$ or is contained in one component of $D_1$, 
\item
no component of $C_1$ contains more than one component of $D_1$ and vice versa.
\end{itemize}
All of these imply that 
the number of components of $C_1$ is equal to the number of components of $D_1$. 
We repeat the same line of arguments for the components of $C_2$ and $D_2$ lying in each component of $C_1$ or $D_1$ separately, 
where we perform the previous algorithm in the larger component which contains the smaller one, 
 and so on, in this way we get the result.
\end{proof}

\begin{rem}
If in (1) in Definition~\ref{adm}
we require having splitting number greater than $0$ instead of having 
algebraic linking number non-zero, then 
the previous arguments could be repeated  
to get a similar result 
if we could prove that 
having two solid tori with splitting number greater than $0$ 
and embedding one circle into each of these tori 
with non-zero winding numbers 
results that the splitting number of these two knots is greater than $0$. 
For similar results about knots and their unknotting numbers, see \cite{ST88, HLP22}. 
\end{rem}

It follows that if two admissible decompositions of $S^3$ are
in the same equivalence class 
of 
the equivalence relation generated by being cylindrically related, 
then they determine the  same sequence of numbers of components. 
So if we define the operation 
$$
(n_0, n_1, \ldots )  +  (m_0, m_1, \ldots ) = (n_0 + m_0, n_1 + m_1, \ldots )
$$
on the set of sequences, then 
the induced map 
$$
[ ( S^3,  \mathcal D, A, C_0 ) ] \mapsto n_{C_0, C_1, C_2, \ldots},$$
where $C_0, C_1, C_2, \ldots$ is some admissible defining sequence, 
is a monoid homomorphism. 

\begin{defn}
For an equivalence class $x$ represented by 
 the admissible decomposition $( S^3,  \mathcal D, A, C_0 )$ and for its admissible defining sequence
 $C_0, C_1, \ldots$ 
let 
$$
L(x) = (l_1, l_2, \ldots)$$
be the sequence of numbers mod $2$ 
of the components of $C_m$ which have non-zero algebraic linking number with some other component of $C_m$.
\end{defn}

\begin{lem}
The map $L$ is well-defined i.e.\ 
admissible decompositions being concordant through finitely many
cylindrically related admissible decompositions have the same value of $L$.
\end{lem}
\begin{proof}
If decompositions with defining sequences $C_0, C_1, \ldots$ and 
$D_0, D_1, \ldots$ 
are cylindrically related, then for every $m \geq 0$ 
 the pairs of components of $C_m$ and the pairs of corresponding components of $D_m$ 
 have the same algebraic linking numbers. Suppose for a decomposition there are two 
 admissible defining sequences 
$C_0, C_1, \ldots$ and 
$D_0, D_1, \ldots$ 
 such that $C_0 = D_0$, we have to show that the linking numbers are equal to $0$ simultanously for both of them (for the components of
 $C_0$ and $D_0$ this is obviously true). Of course we know
 that the components are in bijection with each other 
  by the proof of Proposition~\ref{compnumbsame}
 and in every component of $C_0$
  after some deformation
  we have that 
 \begin{itemize}
\item
the intersection $\del C_1 \cap \del D_1$ is empty,
\item
no component of $C_1$ is disjoint from all the components of $D_1$ and vice versa,
\item
every component of $C_1$ contains one component of $D_1$ or is contained in one component of $D_1$, 
\item
no component of $C_1$ contains more than one component of $D_1$ and vice versa.
\end{itemize}
If a component $T$ of $C_1$ is linked with a component $T'$ of $C_1$
with linking number $0$, then 
any knot in $T'$ is linked with $T$ with linking number $0$.
Also, if a knot in $T'$ is linked with $T$ with linking number $0$, then 
$T$ and $T'$ are linked with linking number $0$.
For every $m \geq 1$ 
after a finite number of iterations
of the algorithm in the proof of Proposition~\ref{compnumbsame} we get the result.
\end{proof}

Of course 
the map $L$
is a monoid homomorphism moreover 
for a class $x$ represented by a slice decomposition 
we have $L(x) = (0, 0, \ldots)$.
 Also, for a class $x$ of the form $[(S^n, \mathcal D, A)] + [(-S^n, \mathcal D, A)]$
we have $L(x) = (0, 0, \ldots)$ since all the linking components appear twice.

\begin{defn}
We call the function 
$$
\nu \co \Delta_3^a \to \Z_2^{\N}$$
obtained by 
$\nu ( [x] ) = L(x)$ 
 the \emph{mod $2$ component number sequence} of the elements of $\Delta_3^a$. 
\end{defn}

\begin{thm}\label{uncount1}
There are at least uncountably many different elements in the concordance group $\Delta^a_3$.
These can be represented by Antoine decompositions. 
\end{thm}
\begin{proof}
For every element $(l_0, l_1, \ldots) \in \Z_2^{\N}$, where
 $l_0 = 0$,
 we have an Antoine decomposition representing a class $x$ such that 
 $\nu([x]) = (l_0, l_1, \ldots)$.
 Hence 
we get uncountably many different classes in the concordance group.
\end{proof}

\subsection{Computations in the cobordism group}\label{computebord}

%
%
%

\begin{prop}\label{manifcob}
Suppose that  $n \geq 0$ and  $M$ is a closed manifold. 
A closed  $n$-dimensional homology manifold $N$
having a resolution   $M \to N$  
is cobordant in $\mathfrak N_n^E$ to $M$.
\end{prop}
\begin{proof}
Take $M \x [0,1]$ and consider the cell-like decomposition $\mathcal D$ of $M$ which results the homology manifold $N$.
 If $\mathcal S(X)$ denotes the collection of singletons in a space $X$, 
then
$\mathcal D \x \mathcal S([0, 1/2])$ union $\mathcal S( M \x (1/2, 1] )$ 
 is a cell-like decomposition of $M \x [0,1]$, denote it by $\mathcal E$.
We have to show that 
the quotient space 
$$M \x [0,1] / \mathcal E$$
is an $(n+1)$-dimensional  homology manifold with boundary homology manifolds
$N$ and $M$.
Take the closed manifold $$M \x [0,1]  \cup_{\va} M \x [0,1],$$
where
$\va \co \del (  M \x [0,1])  \to \del (M \x [0,1])$
is the identity map.
Since $M/ \mathcal D$ is $n$-dimensional, 
the doubling of the decomposition $\mathcal E$ on 
$M \x [0,1]  \cup_{\va} M \x [0,1]$
yields
a finite dimensional quotient space, we get this by using 
estimations for the covering dimension, see \cite{HW41} and \cite[Corollary~2.4A]{Da86}.
So the decomposition space $P$ obtained by factorizing 
$M \x [0,1]  \cup_{\va} M \x [0,1]$
by the double of $\mathcal E$ is a closed  finite dimensional homology manifold by  \cite[Proposition~8.5.1]{DV09}. 
Since this space 
has an open set homeomorphic to $\R^{n+1}$, it is $(n+1)$-dimensional.
We obtain the space $M \x [0,1] / \mathcal E$
by cutting $P$ into two pieces along two subsets homeomorphic to 
$M$ and $N$. This means that $M$ and $N$ are cobordant in $\mathfrak N_n^E$.
\end{proof}

So if every $3$-dimensional  homology manifold is resolvable, then 
$\mathfrak N_3^E = 0$.
Also note that  the decomposition space $S^3 / \mathcal W$ of the Whitehead decomposition $\mathcal W$
is a null-cobordant $3$-dimensional homology manifold, because $[S^3] = 0$.

\begin{prop}\label{cobsubgroup}
For $n \geq 4$ the cobordism  group $\mathfrak N_n$ is a subgroup of $\mathfrak N_n^E$.
\end{prop}
\begin{proof}
Let $M_{1}$ and $M_2$ be closed manifolds. 
If the two cobordism classes $[M_1]$  and $[M_2]$ in 
$\mathfrak N_n^E$ coincide, then 
since $M_i$ are manifolds, we have $i(M_i)=1$ hence 
a cobordism in $\mathfrak N_n^E$ between $M_1$ and $M_2$ also has index $1$ so
this cobordism is resolvable. 
By \cite[Theorem~1.1]{Qu83} and Lemma~\ref{resolcob} there is a manifold cobordism between 
$M_1$ and $M_2$.
\end{proof}


  In Definition~\ref{borddecomp} for $i = 1, 2$ the space $M_i/\mathcal D_i$ is an $n$-dimensional ENR homology manifold and 
$W / \mathcal E$ is an $(n+1)$-dimensional ENR homology manifold if we add the appropriate collars by Lemma~\ref{decomphomolmanif}.  
 If $(M, \mathcal D)$ is such a cell-like decomposition, then 
we can assign the cobordism class of the decomposition space $M /  \mathcal D$
to the cobordism class of $(M, \mathcal D)$. 
This map 
$$\be_n \co \mathcal B_n \to \mathfrak N_n^E$$
$$ [(M, \mathcal D) ] \mapsto [M /\mathcal D]$$
is a group homomorphism.
The image of $\be_n$
contains the classes represented by 
topological manifolds since trivial decompositions always exist
and it contains also the classes represented by homology manifolds having appropriate resolutions.
For $n = 1, 2$  all the homology manifolds are topological manifolds \cite{Wi79}
 so the homomorphism $\be_n$ is 
surjective. 
Take the natural forgetting homomorphism 
$$
F_n \co  \mathcal B_n  \to \mathfrak N_n$$
$$
[(M, \mathcal D) ] \to [M].$$
For every $n \geq 0$ the diagram 
\begin{center}
\begin{graph}(6,2)
\graphlinecolour{1}\grapharrowtype{2}
\textnode {A}(0.5,1.5){$\mathcal B_n$}
\textnode {X}(5.5, 1.5){$\mathfrak N_n^E$}
\textnode {U}(3, 0){$\mathfrak N_n$}
\diredge {A}{U}[\graphlinecolour{0}]
\diredge {U}{X}[\graphlinecolour{0}]
\diredge {A}{X}[\graphlinecolour{0}]
\freetext (3,1.2){$\be_n$}
\freetext (1.2, 0.6){$F_n$}
\freetext (4.6, 0.6){$\va_n$}
\end{graph}
\end{center}
is commutative by Proposition~\ref{manifcob}, where $\va_n$ is the natural map assigning the cobordism class 
$[M] \in \mathfrak N_n^E$ to the cobordism class $[M] \in \mathfrak N_n$.

\begin{prop}\label{deltaimage}
For every $n \geq 0$ 
 the image of $\be_n$ is equal to the 
subgroup of $\mathfrak N_n^E$ generated by the cobordism classes of topological  
manifolds.
\end{prop}
\begin{proof}
The statement follows from the fact that
$F_n$ is surjective. 
\end{proof}

\begin{prop}\label{bordkep}
For $n \geq 1$, we have $\be_n(\mathcal B_n ) = \mathfrak N_n$ in $\mathfrak N_n^E$.
\end{prop}
\begin{proof}
By  Proposition~\ref{cobsubgroup} and Proposition~\ref{deltaimage} 
 we have $\be_n(\mathcal B_n ) = \mathfrak N_n$ for $n \geq 4$.
For $n = 3$, since $\mathfrak N_3 = 0$, 
the statement also holds. For $n = 2$,
The group $\mathfrak N_2$ is isomorphic to  $\Z_2$
so by Proposition~\ref{deltaimage} it is enough to show that 
$\be_2(\mathcal B_2 ) = \Z_2$. But
$[\RP^2]$ is not null-cobordant in $\mathfrak N_2^E$
because $\RP^2$ has a non-zero characteristic number as a smooth or topological manifold and then 
by \cite{BH91} it cannot be null-cobordant.
For $n = 1$, of course $\mathfrak N_1^E = \mathfrak N_1 = 0$.
\end{proof}

\begin{rem}
Instead of cell-like decompositions, which result homology manifolds, it would be possible to study 
decompositions which are just homologically acyclic and nearly $1$-movable, see \cite{DW83}. These result  homology manifolds as well.
 Without being nearly $1$-movable, these can result non-ANR homology manifolds. 
\end{rem}

As we could see, the class $\be_n([(M, \mathcal D)])  = [M/\mathcal D] \in \mathfrak N_n$ 
could  not expose a lot of things about the decomposition $\mathcal D$.
If we add more details to the homology manifolds and their cobordisms,
 then we could obtain a finer invariant of the cobordism group of decompositions. 
 Recall that 
 the singular set 
 of a homology manifold is the set of non-manifold points, which is a closed set.

\begin{defn}[$0$- and $1$-singular homology manifolds]\label{0-1-homolmanif}
A  homology manifold is \emph{$0$-singular} if its singular set  is a 
$0$-dimensional set.
A compact homology manifold with collared boundary is \emph{$1$-singular}  if its singular set $S$
 consists of properly embedded arcs such that $S$ is a direct product 
 in the  collar.  
The  closed $n$-dimensional  $0$-singular 
homology manifolds $X_1$ and $X_2$ are \emph{cobordant} 
if there exists  a compact $(n+1)$-dimensional \emph{$1$-singular} homology manifold $W$ such that 
$\del W$ is homeomorphic to the disjoint union of $X_1$ and $X_2$ and
 $\del W \cap S$ coincides with  the singular set of $X_1 \sqcup X_2$ under this homeomorphism.
 The set of (oriented) cobordism classes is denoted by 
$\mathfrak N_n^S$ (and $\Omega_n^S$).
\end{defn}

The set of cobordism classes
$\mathfrak N_n^S$ and $\Omega_n^S$ are groups with the disjoint union as group operation. 
Denote by $\mathfrak M_n^0$ the cobordism group of $0$-singular manifolds where the cobordisms are arbitrary but 
the singular set of the cobordisms is not the entire manifold.
%


Note that the representatives of the classes in $\be_n(\mathcal B_n)$ are $0$-singular 
and the cobordisms between them 
have not only singular points because the boundary has not only singular points since the singular set is a compact $0$-dimensional set. 
There are  natural homomorphisms
$$
 i_n' \co \mathcal B_n' \to  \mathfrak N_n^S \mbox{, \ \ \ \ \ }i_n \co \mathcal B_n \to  \mathfrak M_n^0 \mbox{\ \ \ \ \ and\ \ \ \ \ } \mathcal B'_n \to \mathcal B_n
 $$
 where $ \mathcal B_n'$ is the version of $\mathcal B_n$ yielding $0$-singular spaces
 and $1$-singular cobordisms, there is 
  the forgetful map 
$$
\va_n \co \mathfrak N_n^S \to \mathfrak M_n^0$$
and then the diagram 
\begin{equation*}
\begin{CD}
 \mathcal B'_n @> i_n' >>  \mathfrak N_n^S    \\
 @VVV @VV \va_n V  \\
\mathcal B_n @> i_n >> \mathfrak M_n^0  @>  \psi_n >> \mathfrak N_n^E 
\end{CD}
\end{equation*}
commutes.
Observe that $\psi_n$ is injective, $\va_n$ is surjective  and since $\be_n ( \mathcal B_n ) = \psi_n \circ i_n ( \mathcal B_n )= \mathfrak N_n$, 
the image $ i_n'(  \mathcal B_n' )$ is in $\va^{-1}_n \circ \psi^{-1}_n(  \mathfrak N_n )$, which could be a larger group
than $\mathfrak N_n$.

%
%
%
%
%
%
%
%


\begin{thebibliography}{AAAAAAA}

\bibitem[ALM68]{ALM68}
S.\ Armentrout, L.\ L.\ Lininger, D.\ V.\  Meyer,
Equivalent decompositions of $\R^3$, 
Pacific J.\ Math.\ {\bf 24} (1968), 205--227.


\bibitem[AS89]{AS89}
 F.\ D.\ Ancel and M.\ P.\ Starbird, The shrinkability of Bing-Whitehead decompositions, 
Topology {\bf 28} (1989), 291--304.


\bibitem[BKKPR21]{BKKPR21}
S.\ Behrens, B.\ Kalm\'ar, M.\ H.\ Kim,\ M.\ Powell, and A.\ Ray, editors, The disc embedding theorem, Oxford University Press, 2021.


\bibitem[BFMW96]{BFMW96}
J.\ Bryant, S.\ Ferry, W.\ Mio, S.\ Weinberger, 
Topology of homology manifolds, 
Ann.\ of Math.\ {\bf 143} (1996), 435--467.



\bibitem[BH91]{BH91}
S.\ Buoncristiano and D.\ Hacon, Characteristic numbers for unoriented $\Z$-homology manifolds, 
Trans.\ Amer.\ Math.\ Soc.\ {\bf 323} (1991), 651--663.


\bibitem[Ca78]{Ca78}
J.\ W.\ Cannon, 
$\Si^2 H^3 = S^5 / G$, Rocky Mountain J.\ Math.\ {\bf 8} (1978), 527--532.


\bibitem[Ca79]{Ca79}
J.\ W.\ Cannon, 
Shrinking cell-like decompositions of manifolds. Codimension three, 
Ann.\ of Math.\ {\bf 110} (1979), 83--112.



\bibitem[Da86]{Da86}
R.\ J.\ Daverman,
Decompositions of manifolds, Academic Press, 1986.

\bibitem[DV09]{DV09}
R.\ J.\ Daverman and 
G.\ A.\ Venema, Embeddings in manifolds, Amer.\ Math.\ Soc., 2009.


\bibitem[DW83]{DW83}
R.\ J.\ Daverman and J.\ J.\ Walsh, 
Acyclic decompositions of manifolds, Pacific J.\ Math.\ {\bf 109} (1983), 291--303.


%

%

%
%
%

\bibitem[Ed80]{Ed80}
 R.\ D.\ Edwards, The topology of manifolds and cell-like maps, in Proceedings of the International Congress of Mathematicians (Helsinki, 1978), 
 Acad.\ Sci.\ Fennica, Helsinki, pages 111--127, 1980.


\bibitem[Ed06]{Ed06}
R.\ D.\ Edwards, Suspensions of homology spheres, e-print, arXiv:math/0610573.


\bibitem[Ed16]{Ed16}
R.\ D.\ Edwards, 
Approximating certain cell-like maps by homeomorphisms, arXiv:1607.08270.
See Notices Amer.\ Math.\ Soc, {\bf 24} (1977), A-649. Abstract \# 751-G5


%


\bibitem[Fr82]{Fr82}
M.\ H.\ Freedman, 
The topology of four-dimensional manifolds, J.\ Diff.\ Geom.\ {\bf 17} (1982), 357--453.

\bibitem[FQ90]{FQ90}
M.\ H.\ Freedman and F.\ Quinn, Topology of $4$-manifolds, Princeton Univ.\ Press, Princeton, N.J., 1990.


\bibitem[GRW\v{Z}11]{GRWZ11}
D.\ Garity, D.\ Repov\v{s}, D.\ Wright and M.\ \v{Z}eljko, Distinguishing Bing-Whitehead Cantor sets, 
Trans.\ Amer.\ Math.\ Soc.\ {\bf 363} (2011), 1007--1022.


\bibitem[Ha51]{Ha51}
O.\ Hanner, Some theorems on absolute neighborhood retracts, Ark.\ Mat.\ {\bf 1} (1951),  389--408.

%



\bibitem[HLP22]{HLP22}
J.\ Hom, T.\ Lidman and J.\ Park,
Unknotting number and cabling, arxiv:2206.04196.



\bibitem[HW41]{HW41}
W.\ Hurewicz and H.\ Wallman, Dimension theory, 1941. 


%


\bibitem[Jo99]{Jo99}
H.\ Johnston,
Transversality for homology manifolds, Topology {\bf 38} (1999), 673--697.



\bibitem[JR00]{JR00}
H.\ Johnston and A.\ Ranicki, Homology manifold bordism, Trans.\ Amer.\ Math.\ Soc.\ {\bf 352} (2000), 5093--5137.










%
%
%
%
%

\bibitem[KP14]{KP14}
D.\ Kasprowski and M.\ Powell, Shrinking of toroidal decomposition spaces, Fundamenta Mathematicae {\bf 227} (2014), 271--296.


\bibitem[KT76]{KT76}
L.\ H.\ Kauffman and L.\ R.\ Taylor, Signature of links, Trans.\ Amer.\ Math.\ Soc.\ {\bf 216} (1976), 351--365.





\bibitem[Ma71]{Ma71}
C.\ R.\ F.\ Maunder, On the Pontrjagin classes of homology manifolds, Topology {\bf 10} (1971), 111--118.




\bibitem[Mi90]{Mi90}
W.\ J.\ R.\ Mitchell, Defining the boundary of a homology manifold, Proc.\ Amer.\ Math.\ Soc.\ {\bf 110} (1990), 509--513.


\bibitem[Mu65]{Mu65}
K.\ Murasugi, On a certain numerical invariant of link types, Trans.\ Amer.\ Math.\ Soc.\ {\bf 117} (1965), 387--422.

%



\bibitem[Qu82]{Qu82}
F.\ Quinn, Ends of maps.\ III.\ Dimensions 4 and 5, 
J.\ Differential Geometry {\bf 17} (1982), 503--521.


\bibitem[Qu83]{Qu83}
F.\ Quinn, Resolutions of homology manifolds, and the topological characterization of manifolds, 
Invent.\ Math.\ {\bf 72} (1983), 267--284.

\bibitem[Qu87]{Qu87}
F.\ Quinn, An obstruction to the resolution of homology manifolds, Michigan Math.\ J.\ {\bf 34} (1987), 285--291.



%
%

\bibitem[ST88]{ST88}
M.\ G.\ Scharlemann and A.\ Thompson, Unknotting number, genus, and companion tori, Math.\ Ann.\ {\bf 280} (1988), 191-- 205.


%



\bibitem[Sh68]{Sh68b}
R.\ B.\ Sher, 
Concerning wild Cantor sets in $E^3$, 
Proc.\ Amer.\ Math.\ Soc.\ {\bf 19} (1968), 1195--1200.





\bibitem[St68]{Sto}
{R. E. Stong},
Notes on cobordism theory,
Princeton University Press and the University of Tokyo Press, Princeton, 1968.


\bibitem[Th84]{Th84}
T.\ L.\ Thickstun, An extension of the loop theorem and resolutions
of generalized $3$-manifolds with $0$-dimensional singular set, Invent.\ Math.\ {\bf 78} (1984), 161--222.


\bibitem[Th04]{Th04}
T.\ L.\ Thickstun, Resolutions of generalized $3$-manifolds whose singular sets have general position dimension one, 
Topology Appl.\ {\bf 138} (2004), 61--95.

%
%
%


\bibitem[Wi79]{Wi79}
R.\ L.\ Wilder, Topology of Manifolds, Amer.\ Math.\ Soc.\ Colloq.\ Publ., vol. 32, Amer.\ Math.\
Soc., Providence, R.I. Reprint of the 1963 edition, 1979.



\bibitem[\v{Z}e05]{Ze05}
M.\ \v{Z}eljko, Genus of a Cantor set, Rocky Mountain J.\ Math.\ {\bf 35} (2005), 349--366.




\end{thebibliography}
\end{document}